\documentclass[12pt]{amsart}
\usepackage{amssymb,amsfonts,latexsym,amscd}

\usepackage[all,cmtip,2cell]{xy}
\UseAllTwocells
\usepackage{verbatim}
\newtheorem{theorem}{Theorem}[section]
\newtheorem{lemma}[theorem]{Lemma}
\newtheorem{proposition}[theorem]{Proposition}
\newtheorem{corollary}[theorem]{Corollary}
\theoremstyle{definition}

\newtheorem{example}[theorem]{Example}

\newtheorem{conjecture}[theorem]{Conjecture}
\newtheorem{remark}[theorem]{Remark}

\newcommand{\h}{\mathfrak{h}}

\newcommand{\ben}{\begin{enumerate}}
\newcommand{\een}{\end{enumerate}}

\newcommand{\p}{\partial}

\renewcommand{\aa}{{\bold a}}

\newcommand{\nad}[2]{\genfrac{}{}{0pt}{}{#1}{#2}}
\def \a {\alpha}

\begin{document}

\title[Algebras generated by generalized power sums] {On Cohen-Macaulayness of algebras generated by generalized power sums}

\author{Pavel Etingof}
\address{Department of Mathematics, Massachusetts Institute of Technology,
Cambridge, MA 02139, USA} \email{etingof@math.mit.edu}

\author{Eric Rains}
\address{Department of Mathematics, California Institute of
Technology, Pasadena, CA 91125, USA}\email{rains@math.caltech.edu}

\maketitle

\centerline{\large with an appendix by Misha Feigin} 
\vskip .1in

\vskip .05in
\centerline{\bf To Sasha Veselov on his 60-th birthday, with admiration} 
\vskip .05in

\section{Introduction} 

Let $a_{ij}, i\ge 1, 1\le j\le N$, be nonzero complex numbers. 
Let 
$$
Q_i(x_1,...,x_N)=\sum_{j=1}^Na_{ij}x_j^i.
$$ 
We call the functions $Q_i$ {\it generalized power sums}. 
Let $A$ be the algebra generated by $Q_i, i\ge 1$ inside 
$\Bbb C[x_1,...,x_N]$. For generic $a_{ij}$, this algebra is finitely generated. 
The main question studied in this paper is when the algebra $A$ is Cohen-Macaulay (shortly, CM). 
Specifically, following \cite{BCES}, for various collections of positive integers $(r_1,...,r_k)$ with $\sum r_i=N$, we 
study the CM property of algebras of generalized power sums with symmetry type $(r_1,...,r_k)$
(i.e. symmetric in the first $r_1$ variables, the next $r_2$ variables, etc.).  

In Section 2, we study the simplest nontrivial case -- type $(1,1)$. In this case, by renormalizing $Q_i$, we can assume that 
$a_{i1}=a_i$ and $a_{i2}=1$, so $Q_i=a_iy^i+z^i$. We show that if $a_1,a_2,a_3$ are generic, then 
$A$ is CM if and only if $a_i=c^i\frac{q^i-1}{1-t^i}$ for some numbers $c,q,t\in \Bbb C$. 

In Section 3, we extend this analysis to the case of type $(r,s)$. 
Namely, we show that $A$ is CM if $a_i=c^i\frac{q^i-1}{1-t^i}$, i.e., after rescaling the variables $y_j$, 
$$
Q_i=\frac{q^i-1}{1-t^i}(y_1^i+...+y_r^i)+(z_1^i+...+z_s^i)
$$
In this case, the algebra $A$ is 
the algebra of $q,t$-deformed Newton sums introduced by Sergeev and Veselov in \cite{SV2}. 
If $t=q^{-n}$, where $n$ is a positive integer, this is a subalgebra of the algebra of quantum integrals 
of the deformed Macdonald-Ruijsenaars system. Our proof of the CM property of this algebra is based on degeneration to the classical case,
$$
Q_i=a(y_1^i+...+y_r^i)+(z_1^i+...+z_s^i)
$$
(obtained by setting $q=t^{-a}$ and tending $t$ to $1$) 
where the CM property is established in \cite{BCES} based on the methods of \cite{EGL} (namely, the representation 
theory of rational Cherednik algebras with minimal support). 

In Section 4, we study the case of type $(1,r,s)$. We show that in this case the CM property occurs 
generically for the generalized power sums 
$$
\frac{q^i-t^i}{1-t^i}x^i+\frac{q^i-1}{1-t^i}(y_1^i+...+y_r^i)+(z_1^i+...+z_s^i).
$$
and their degenerations
$$
(a+1)x^i+a(y_1^i+...+y_r^i)+(z_1^i+...+z_s^i) 
$$
(again obtained by setting $q=t^{-a}$, $t\to 1$). 
Namely, we prove this by reduction to type $(r+1,s+1)$. 
For $r=1$, this confirms the first statement of Conjecture 7.4 in \cite{BCES}.

We also show that in all of the above cases, the CM algebras $A$ can be defined by quasi-invariance conditions on hyperplanes. 
In the case $(1,r,s)$, these quasi-invariance conditions appear to be new. 

In Section 5 we use a similar method to the one of Section 4 to 
prove that for any $m\ge 1$, $n\ge 3$, the union of $S_{mn}$-translates 
of the subspace
$$
x_1=...=x_{2m},\ x_{2m+1}=\dots=x_{3m},\ \dots,x_{(n-1)m+1}=\dots=x_{nm}
$$
(i.e., one group of $2m$ equal coordinates and $n-2$ groups of $m$ equal coordinates) 
is CM. This is done by reducing to the case of $n$ $m$-tuples of equal coordinates, where the result is 
proved using representations of rational Cherednik algebras in \cite{EGL}. 

In Section 6, we apply the theory of representations of the rational Cherednik algebra of minimal support to $m$-quasi-invariants considered in \cite{FV1,FV2}. 

In the Appendix, M. Feigin uses this approach to prove a conjecture from \cite{FV2} that the algebra of $m$-quasi-invariants in the case of one light particle ($s=1$) 
is Gorenstein. 

In particular, this paper explains all the instances of CM algebras found 
experimentally in \cite{BCES}, and confirms the philosophy (originating from \cite{FV1} and developed further in \cite{BEG})
that the CM property of algebras of this type should be rare, and, 
whenever it occurs, should be related to quasi-invariance conditions on 
hyperplanes, quantum integrable systems, and ultimately to representation theory. 

{\bf Acknowledgements.} The work of P.E. was  partially supported by the NSF grant DMS-1000113. 
The authors are grateful to O. Chalykh, M. Feigin, and A. Veselov for useful discussions. 

\section{Type $(1,1)$} 

\subsection{The algebra $\Lambda_\aa$} Let $\aa=(a_1,a_2,...)$ be a sequence of nonzero complex numbers.
Let $\Lambda_\aa$ be the subalgebra of $\Bbb C[y,z]$ generated by the polynomials
$Q_{i,a_i}$, $i\ge 1$, where 
$$
Q_{i,a}:=ay^i+z^i. 
$$
(When no confusion is possible, we will denote $Q_{i,a_i}$ simply by $Q_i$.) 

We will be interested in the question when the algebra $\Lambda_{\aa}$ is CM. 
Note that by renormalizing $y$, we may replace $a_i$ by $a_i/a_1^i$, and thus assume that $a_1=1$. 

\subsection{Auxiliary lemmas} 

We need a few auxiliary lemmas. 

\begin{lemma}\label{fingen} 
If $a_2+a_1^2\ne 0$ then the algebra $\Lambda_\aa$ 
is finitely generated as a module over the polynomial algebra $\Bbb C[Q_1,Q_2]$
(in particular, as a ring). 
\end{lemma}

\begin{proof} We may assume that $a_1=1$. 
It suffices to show that the equations $Q_1(y,z)=0$, $Q_2(y,z)=0$, i.e.
$$
y+z=0,\ a_2y^2+z^2=0
$$ 
have only the zero solution (then the entire polynomial algebra $\Bbb C[y,z]$ is finite over $\Bbb C[Q_1,Q_2]$, 
so $\Lambda_\aa$ is as well, by the Hilbert basis theorem). 
From the first equation we get $z=-y$, and substituting this into the second one, we get $(a_2+1)y^2=0$. 
Since $a_2\ne -1$, we have $y=z=0$.  
\end{proof} 

\begin{lemma}\label{freesubmod} Let $a_2\ne -a_1^2$, and $(a_2,a_3)\ne (a_1^2,a_1^3)$. 
Let $M\subset \Lambda_{\aa}$ be the submodule over $\Bbb C[Q_1,Q_2]$ generated by $1$ and $Q_3$. Then $M$ is free of rank $2$, so its Hilbert series is 
$$
h(u)=\frac{1+u^3}{(1-u)(1-u^2)}.
$$
\end{lemma} 

\begin{proof} We may assume that $a_1=1$. First, we claim that $Q_3\notin \Bbb C[Q_1,Q_2]$. 
Assume the contrary, that $Q_3=\alpha Q_1^3+\beta Q_1Q_2$. 
Then from comparing coefficients we have 
$$
\alpha+\beta a_2=a_3,\ 3\alpha+\beta a_2=0,\ \alpha+\beta=1,\ 3\alpha+\beta=0,
$$ 
which implies that $(a_2,a_3)=(1,1)$, a contradiction.

 Since by Lemma \ref{fingen}, $\Lambda_\aa$ is a finitely generated 
 $\Bbb C[Q_1,Q_2]$-module, this implies that $Q_3\notin \Bbb C(Q_1,Q_2)$. 
Thus, the module $M$ is indeed free with the stated Hilbert series. 
\end{proof} 

Now suppose that the assumptions of Lemma \ref{freesubmod} are satisfied. 
It is clear that the rank of $\Lambda_\aa$ over $\Bbb C[Q_1,Q_2]$ is $2$. 
Thus, $\Lambda_\aa$ is CM if and only if it coincides with $M$. Hence, since $\frac{1}{(1-u)^2}-h(u)=\frac{u}{1-u}$, 
we obtain the following lemma. 

\begin{lemma}\label{cmequiv} Under the assumptions of Lemma \ref{freesubmod}, 
the CM property of $\Lambda_\aa$ is equivalent to saying that the codimension of $\Lambda_\aa[i]$ in homogeneous polynomials 
of degree $i$ is $1$ for each $i\ge 1$.
\end{lemma} 

Note that this codimension is clearly at most $1$. 

\subsection{The CM property of $\Lambda_\aa$} 

Let $q,t$ be not roots of unity, $t\ne q,q^{-1}$, and $c\ne 0$.  

\begin{theorem}\label{CMnesscrit} (i) If $a_i=c^i\frac{q^i-1}{1-t^i}$ for $i\ge 1$ then the algebra $\Lambda_\aa$ is CM with Hilbert series 
$h(u)$. 

(ii) Let $\aa$ be any sequence of nonzero numbers, and $c,q,t$ be such that $a_i=c^i\frac{q^i-1}{1-t^i}$ 
for $i=1,2,3$. Assume that $q,t$ are not roots of unity, and $t\ne q,q^{-1}$. If $\Lambda_\aa$ 
is CM, then $a_i=c^i\frac{q^i-1}{1-t^i}$ for all $i\ge 1$. 

(iii) Let $a_i=c^ia$, where $a\ne \pm 1$. Then the algebra $\Lambda_\aa$ is CM with Hilbert series 
$h(u)$. 

(iv) If $a_i=c^ia$ with $a\ne \pm 1$ for $i=1,2,3$, and if $\Lambda_\aa$ is CM, then $a_i=c^ia$ for all $i\ge 1$. 
\end{theorem} 

\begin{remark} It is easy to show that for generic $a_1,a_2,a_3$ the equations 
$$
a_i=c^i\frac{q^i-1}{1-t^i}, i=1,2,3
$$ 
lead to a quadratic equation, and thus have two solutions $(c,q,t)$, related by the Galois symmetry 
$(c,q,t)\to (cqt^{-1},q^{-1},t^{-1})$. In particular, for generic $a_1,a_2,a_3$ a solution 
$(c,q,t)$ exists, and Theorem \ref{CMnesscrit}(ii) applies.  
\end{remark} 

\begin{proof} By renormalizing $y$, we may assume without loss of generality that $c=1$. 
Let us make this assumption. 

(i) Any element $f\in \Lambda_\aa$ satisfies the quasi-invariance condition 
$$
f(tx,qx)=f(x,x). 
$$
Indeed, this condition is satisfied for each generator $Q_i$, and if it is satisfied for $f$ and $g$ then it is satisfied for $fg$. 
This gives a codimension $1$ subspace in $\Bbb C[y,z][i]$ for all $i\ge 1$ (since the function $z^i$ does not satisfy this condition, as $t$ is not a root of unity). So by Lemma \ref{cmequiv}, the result holds under the assumptions of Lemma \ref{freesubmod}.
In terms of $q$ and $t$, these assumptions turn into the conditions that $qt\ne 1$ and $q\ne t$, so they are satisfied. 

(iii) is a limiting case of (i) (for $q=t^{-a}$ and $t\to 1$), so in this case $f\in \Lambda_\aa$ 
satisfies the limiting quasi-invariance condition 
$$
((a\partial_2-\partial_1)f)(x,x)=0,
$$
giving a codimension $1$ condition in each positive degree. The assumptions of Lemma \ref{freesubmod}
in this case turn into the conditions $a\ne \pm 1$, so they are satisfied, and Lemma \ref{cmequiv} implies the statement. 

(ii) Let $i\ge 4$. By Lemma \ref{freesubmod}, homogeneous polynomials of degree $i$ in $Q_1,Q_2,Q_3$ (linear in $Q_3$) 
span a subspace of codimension $1$ in $\Bbb C[y,z][i]$ --- the space of solutions of the quasi-invariance equation 
$f(tx,qx)=f(x,x)$. So $Q_i$ must also satisfy this condition. Thus, 
$t^ia_i+q^i=a_i+1$, i.e., $a_i=\frac{q^i-1}{1-t^i}$, as desired. 

(iv) The proof is similar to (ii), except that we use the limiting quasi-invariance condition $((a\partial_2-\partial_1)f)(x,x)=0$.
\end{proof} 

\begin{remark} In spite of Theorem \ref{CMnesscrit}, there exist 
infinite-parameter families of sequences $\aa$ for which 
$\Lambda_\aa$ is CM. For instance, if
  $q$ and $t$ are primitive $n$th roots of unity with $t\ne q,q^{-1}$,
  then for any sequence with $a_i=c^i (q^i-1)/(1-t^i)$ when $i$ is not
  a multiple of $n$, the corresponding $\Lambda_\aa$ is CM.  Indeed, in
  that case, the algebra generated by $Q_1$, $Q_2$, $Q_3$ is
  determined by the same (codimension 1) quasi-invariance condition as
  in the generic case, but since both $x^{mn}$ and $y^{mn}$ are
  quasi-invariant, any linear combination of them is contained in the algebra.
(Note that in this case $\Lambda_\aa$ does not actually depend on $a_n,a_{2n},...$). 
\end{remark}

\section{Type $(r,s)$} 

\subsection{Finite generation} 
First let us prove a general result on finite generation
(which is fairly standard, see e.g. \cite{SV2}, Theorem 5.1).   
Let $a_{ij}, i\ge 1, 1\le j\le N$, be nonzero complex numbers. 
Let $Q_i(x_1,...,x_N)=\sum_{j=1}^Na_{ij}x_j^i$. Let $A$ be the algebra 
generated by $Q_i, i\ge 1$ inside $\Bbb C[x_1,...,x_N]$. 

\begin{proposition}\label{fingen1} $A$ is finitely generated if and only if the system of equations 
\begin{equation}\label{syst} 
Q_i(x_1,...,x_N)=0,\ i\ge 1
\end{equation} 
has only the zero solution. 
\end{proposition}

\begin{proof}
Suppose the system (\ref{syst}) has only the zero solution. By the Hilbert basis theorem, 
there is $k\ge 1$ such that this is true already for the first $k$ equations.
Then $\Bbb C[x_1,...,x_N]$ is a finitely generated module over $\Bbb C[Q_1,...,Q_k]$, 
and hence, by the Hilbert basis theorem, so is $A$. Thus, $A$ is finitely generated as an algebra. 

Conversely, suppose \eqref{syst} has a nonzero solution $(x_1^*,...,x_N^*)$.  
Without loss of generality we can assume that $x_1^*\ne 0$. 
Let $x_1=yx_1^*$ and $x_i=zx_i^*$ for $i \ge 2$, where 
$y,z$ are new variables. Then $Q_i$ specialize to 
$Q_i^*(y,z)=a_{i1}x_1^{*i}(y^i-z^i)$. So we just need to show that the 
algebra generated by the polynomials $f_i(y,z):=y^i-z^i$ is not finitely generated. 
But this is easy and well known (see e.g. \cite{BCES}, Remark 2.7(3)).  
\end{proof} 

\subsection{The algebra $\Lambda_{r,s,\aa}$ and its CM properties.}

Now let $r,s\ge 1$ be integers, and 
$$
Q_{r,s,i,a}(y_1,...,y_r,z_1,...,z_s):=a(y_1^i+...+y_r^i)+z_1^i+...+z_s^i. 
$$
Let $\aa=(a_1,a_2,...)$ be a sequence of nonzero complex numbers, 
and $\Lambda_{r,s,\aa}$ be the subalgebra of $\Bbb C[y_1,...,y_r,z_1,...,z_s]$ 
generated by $Q_{r,s,i,a_i}$ for all $i\ge 1$. When no confusion is possible, we 
will denote $Q_{r,s,i,a_i}$ simply by $Q_i$. 

By Proposition \ref{fingen1}, $\Lambda_{r,s,\aa}$ is finitely generated
if and only if the system
\begin{equation}\label{twobran3} 
a_i(y_1^i+...+y_r^i)+z_1^i+...+z_s^i=0, i\ge 1
\end{equation} 
has only the zero solution. If $r=1$, this implies 
that the algebra is infinitely generated iff $a_i=-(\beta_1^i+...+\beta_s^i)$ for some $\beta_1,...,\beta_s\in \Bbb C$, 
and a similar statement holds for $a_i^{-1}$ for $s=1$. However, if 
$r,s\ge 2$, the set of sequences violating finite generation is infinite dimensional. 
For example, taking $y_1=1,y_2=-1,z=1,z_2=-1$, and the rest of 
$y_j,z_l$ to be zero, we get that the sequence 
with $a_i=-1$ for odd $i$ and $a_i$ arbitrary for even $i$ 
violates finite generation. 

We would like to know when $\Lambda_{r,s,\aa}$ is CM. 
Note that as before, we may assume that $a_1=1$ (or any other nonzero constant) 
by renormalizing $y_i$. 

Our first result is the following theorem. 

Let $c\ne 0$, $q,t$ be not roots of unity, and 
$a_i=c^i\frac{q^i-1}{1-t^i}$. For $c=t^i$, this 
is the algebra of $q,t$-deformed Newton sums 
 (see \cite{SV2}, Section 5).

\begin{theorem}\label{CMnessrs} (i) (\cite{SV2}, Theorem 5.1) $\Lambda_{r,s,\aa}$ is finitely generated if and only if 
$q^m\ne t^n$ for integers $1\le m\le r,1\le n\le s$.  

(ii) If $q,t$ are Weil generic (i.e., outside of a countable union of curves) then $\Lambda_{r,s,\aa}$ is CM with Hilbert series 
$$
h_{r,s}(u)=\frac{1}{(u,u)_r}\sum_{i=0}^s\frac{u^{i(r+1)}}{(u,u)_i}.
$$
where $(u,u)_m:=(1-u)....(1-u^m)$. 
\end{theorem}

Note that Theorem \ref{CMnessrs} is a generalization (or, more precisely, a deformation) of the following theorem. 

Let $\Lambda_{r,s,a}$ be the algebra corresponding to the sequence $a_i=a$.

\begin{theorem}\label{classi} (\cite{SV1}, Theorem 5) $\Lambda_{r,s,a}$ is finitely generated if and only if 
$a\ne -n/m$ for integers $1\le m\le r,1\le n\le s$.  

(ii) (\cite{BCES}, Theorem 4.4) For generic $a$ the algebra $\Lambda_{r,s,a}$ is CM with Hilbert series $h_{r,s}(u)$.  
\end{theorem} 

\begin{remark} 1.  By analogy with Conjecture 4.8 of \cite{BCES}, we expect that the exceptional set for Theorem \ref{CMnessrs}(ii) is 
$q^m=t^{\pm n}$, where $1\le m\le r$, $1\le n\le s$ (assuming $q,t\ne 0$ and are not roots of unity).

2. The formula for the Hilbert series in Theorem \ref{classi}(ii) is given in \cite {SV1} and in the $q,t$-case in \cite{SV2}.  
\end{remark} 

\begin{proof} Without loss of generality, we may assume that $c=1$. 

(i)  This is proved in \cite{SV2}, but we reproduce the proof for reader's convenience. 
Consider the system of equations $Q_i=0$, $i\ge 1$. 
It can be written as 
\begin{equation}\label{sys1} 
\sum_{j=1}^r (qy_j)^i+\sum_{l=1}^s z_l^i=\sum_{j=1}^r y_j^i+\sum_{l=1}^s (tz_l)^i. 
\end{equation} 
Suppose that this system has a nontrivial solution. Let $m$ be the number of nonzero coordinates $y_j$ and $n$ be the number of nonzero coordinates
$z_l$ in this solution. Since \eqref{sys1} 
holds for each $i$, each nonzero term on the LHS must equal some nonzero term on the RHS. By taking products, this implies that 
$q^m=t^n$. Note that $m,n>0$ since $m+n>0$ and $q,t$ are not roots of $1$. Conversely, suppose $q^m=t^n$.
If $q=t=0$, then \eqref{sys1} has a nonzero solution $y_1=z_1=1$, $y_j=z_l=0$ for $j,l\ge 2$. If $q,t$ are not both zero, then  taking
$y_1=t^n, y_2=t^nq,...,y_n=t^nq^{m-1}$, $z_1=q^mt^{n-1},...,z_{n-1}=q^mt, z_n=q^m$, and the rest of $y_j$ and $z_l$ to be zero, we also obtain a nontrivial solution. Thus, the result follows from Proposition \ref{fingen1}.  

(ii) Let $q=t^{-a}$ and $t\to 1$. Then $\Lambda_{r,s,\aa}$ degenerates to $\Lambda_{r,s,a}$. 
 By Theorem \ref{classi}(ii), for generic $a$, the algebra $\Lambda_{r,s,a}$ is a free module 
of finite rank over $\Bbb C[Q_1,...,Q_{r+s}]$, with Hilbert series $h_{r,s}(u)$. Thus, our job is to show that 
for Weil generic $q,t$, the Hilbert series of $\Lambda_{r,s,\aa}$ is dominated by $h_{r,s}(u)$ coefficientwise
(this will imply that it actually equals to $h_{r,s}(u)$). 

This is proved in \cite{SV2}, Section 5, and we reproduce the proof for reader's convenience. 
Let $\Lambda$ be the ring of symmetric functions, and define a surjective homomorphism $\phi: \Lambda\to \Lambda_{r,s,\aa}$ 
given by the formula $\phi(p_i)=Q_i$, where $p_i$ are the power sums. By Theorem 5.6 of \cite{SV2}, for generic $q,t$, 
${\rm Ker}\phi$ has a basis consisting of 
Macdonald polynomials $P_\lambda$, where $\lambda$ is a Young diagram that does not fit into the fat $(r,s)$-hook 
(i.e., $\lambda_{r+1}>s$), while $\Lambda_{r,s,\aa}$ has a basis formed by $P_\lambda(q,t)$ for $\lambda$ fitting into the $(r,s)$-hook (i.e., $\lambda_{r+1}\le s$), with the Hilbert series $h_{r,s}(u)$. This means that the kernel does not shrink as we deform the limiting case to Weil generic $q,t$, as desired. 
\end{proof} 

If $a_i=\frac{q^i-1}{1-t^i}$, we will denote the algebra $\Lambda_{r,s,\aa}$ by $\Lambda_{r,s,q,t}$. 

\subsection{The quasi-invariance conditions} 
Below we will use the following proposition, due to Sergeev and Veselov. 

\begin{proposition}\label{quasi-inva} (i) (\cite{SV2}) If $q,t\in \Bbb C^\times$ are not roots of unity, and $q^m\ne t^n$ for $n,m\ge 1$ then 
$\Lambda_{r,s,q,t}$ for $a_i=\frac{q^i-1}{1-t^i}$ is the algebra of symmetric polynomials in $y_j$ and in $z_l$ satisfying the quasi-invariance conditions
\begin{equation}\label{qquas}
f(y_1,...,ty_j,...,y_r,z_1,...,qz_l,...,z_s)=f(y_1,...,y_j,...,y_r,z_1,...,z_l,...,z_s),
\end{equation}
when $y_j=z_l$ for all $j\in [1,r],l\in [1,s]$. 

(ii) (\cite{SV1}) 
If $a$ is generic then $\Lambda_{r,s,a}$ is the algebra of symmetric polynomials in $y_j$ and in $z_l$ satisfying the quasi-invariance conditions
\begin{equation}\label{clquas}  
((a\partial_{z_l}-\partial_{y_j})f)(y_1,...,y_j,...,y_r,z_1,...,z_l,...,z_s)=0,
\end{equation} 
when $y_j=z_l$ for all $j\in [1,r],l\in [1,s]$. 
\end{proposition}

\section{Type $(1,r,s)$}

\subsection{The result} 
As before, let $q,t\in \Bbb C^\times$ be not roots of unity such that $q\ne t$. Let $r,s$ be positive integers.  
Consider the polynomials 
$$
P_{r,s,i,q,t}:=\frac{q^i-t^i}{1-t^i}x^i+\frac{q^i-1}{1-t^i}(y_1^i+...+y_r^i)+(z_1^i+...+z_s^i).
$$
Let $A_{r,s, q,t}$ be the algebra generated by the $P_{r,s,i,q,t}$, $i\ge 1$. 

We will also be interested in the limiting case $q=t^{-a}$, $t\to 1$. In this limit, 
we get the polynomials 
$$
P_{r,s,i,a}:=(a+1)x^i+a(y_1^i+...+y_r^i)+(z_1^i+...+z_s^i) 
$$
Let $A_{r,s,a}$ be the algebra generated by the $P_{r,s,i,a}$, $i\ge 1$. 

In both cases, when no confusion is possible, 
we will denote the generating polynomials simply by $P_i$. 

Note that if $a_i=\frac{q^i-1}{1-t^i}$ then 
the restriction of $Q_{r+1,s+1,i,a_i}$ to the hyperplane $y_{r+1}=z_{s+1}$
is $P_{r,s,i,q,t}$, where $x=y_{r+1}=z_{s+1}$. Similarly, the restriction of 
$Q_{r+1,s+1,i,a}$ is $P_{r,s,i,a}$. Thus, we have an epimorphism 
$\phi_{q,t}: \Lambda_{r+1,s+1,q,t}\to A_{r,s,q,t}$, which degenerates 
to an epimorphism $\phi_a: \Lambda_{r+1,s+1,a}\to A_{r,s,a}$.  

\begin{theorem}\label{CMness1rs}  
(i) The algebra $A_{r,s,a}$ is CM for generic $a$. Moreover, the Hilbert series of this algebra is given by the formula 
$$
h_{A_{r,s,a}}(u)=h_{\Lambda_{r+1,s+1,a}}(u)-\frac{u^{2(r+1)(s+1)}}{(u,u)_{r+1}(u,u)_{s+1}}.
$$

(ii) For Weil generic $q,t$, the algebra $A_{r,s,q,t}$ is CM with the same Hilbert series. 
\end{theorem} 

In the special case $r=1$, 
this confirms the first part of Conjecture 7.4 in \cite{BCES}. 

A proof of Theorem \ref{CMness1rs} is given in the next subsection. 

\subsection{Proof of Theorem \ref{CMness1rs}}
We will need the following simple lemma from homological algebra. 
     
\begin{lemma}\label{homalg} 
Let $C$ be a commutative algebra, $I$ an ideal in $C$, and $C'$ a subalgebra of $C$ containing $I$.
Let $B\subset C'$ be a subalgebra such that $C,C',C/I$ are all projective modules over $B$.   Then so is $C'/I$.
\end{lemma} 

\begin{proof} The short exact sequence
$$
0\to C'\to C\to C/C'\to 0
$$
is a $B$-projective resolution of $C/C'$, which therefore has homological dimension $\le 1$.  
Since $C/I$ is $B$-projective, the short exact sequence
$$
0\to C'/I\to C/I\to C/C'\to 0
$$
must also be a projective resolution, and thus $C'/I$ is projective.
\end{proof} 

We will apply Lemma \ref{homalg} in the following situation: 
$$
C=\Bbb C[y_1,...,y_{r+1},z_1,...,z_{s+1}],\ C'=\Lambda_{r+1,s+1,a},\ I={\rm Ker}\phi_a.
$$ 
For this, we need to prove another auxiliary lemma. 

\begin{lemma}\label{ideal} $I$ is an ideal in $C$. 
More precisely, $I$ is the principal ideal generated 
by the polynomial
$$
D_{r+1,s+1}(\bold y,\bold z):=\prod_{j=1}^{r+1}\prod_{l=1}^{s+1}(y_j-z_l)^2,
$$
and thus its Hilbert series is given by the formula 
$$
h_I(u)=\frac{u^{2(r+1)(s+1)}}{(u;u)_{r+1}(u;u)_{s+1}},
$$
\end{lemma} 

\begin{proof} By Proposition \ref{quasi-inva}(ii), $C'$ 
is the algebra of polynomials symmetric in $y_j$ and $z_l$ and 
satisfying the quasi-invariance condition
$$
((a\partial_{y_j}-\partial_{z_l})f)(y_1,...,y_j,...,y_{r+1},z_1,...,z_l,...,z_{s+1})=0
$$
when $y_j=z_l$ for all $j\in [1,r+1],l\in [1,s+1]$. 
This implies that $D_{r+1,s+1}C\subset I\subset C'$ (as the restriction of 
$D_{r+1,s+1}$ to the hyperplane $y_{r+1}=z_{s+1}$ is zero, and any multiple of 
$D_{r+1,s+1}$ satisfies the quasi-invariance condition). Also, if $f\in I$ then its restriction to the hyperplane 
$y_j=z_l$ is zero and it satisfies the quasi-invariance condition, so must be divisible by $(y_j-z_l)^2$. 
Thus by symmetry $f$ is divisible by $D_{r+1,s+1}$. Thus $f\in D_{r+1,s+1}C$ and $D_{r+1,s+1}C=I$.
This implies all statements of the lemma. 
\end{proof} 

Now we prove part (i) of the theorem. 
To apply Lemma \ref{homalg}, we will now define $B:=\Bbb C[Q_1,...,Q_{r+s+1}]$
(where $Q_i:=Q_{r+1,s+1,i,a}$). Then $C$ is clearly free over $B$ (of infinite rank), as it is free 
of finite rank over $\Bbb C[Q_1,...,Q_{r+s+2}]$ by Serre's theorem (since 
$C$ is a polynomial algebra). Also, $C'$ is free over $B$ (of infinite rank), as 
it is free of finite rank over $\Bbb C[Q_1,...,Q_{r+s+2}]$ by 
Theorem \ref{classi}(ii) (since $C'$ is a CM algebra). Finally,  
$C/I$ is CM (as it is the ring of functions on a hypersurface). So to show that $C/I$ is 
free over $B$, it suffices to show that it is finitely generated as a module, i.e., the system 
of equations 
$$
Q_{r,s,i,a}(\bold y,\bold z)=0,\ i=1,...,r+s+1;\ D_{r+1,s+1}(\bold y,\bold z)=0
$$
has only the zero solution. By symmetry we may assume that $y_{r+1}=z_{s+1}$, so,  
substituting, we get 
$$
P_i(x,\bold y,\bold z)=0,\ i=1,...,r+s+1,
$$
which we know has only the zero solution (see \cite{BCES}, proof of Proposition 2.6). 
Thus, by Lemma \ref{homalg}, $C'/I=A_{r,s,a}$ is
a free module over $B$. It is also a finitely generated module. 
This implies that $A_{r,s,a}$ is a CM algebra with the claimed Hilbert series, as desired. 

Let us now prove part (ii) of the theorem. Since the algebra $A_{r,s,q,t}$ 
is generated by polynomials which deform the polynomials generating $A_{r,s,a}$, 
it suffices to show that the Hilbert series $h_{A_{r,s,q,t}}(u)$ is dominated coefficientwise by the Hilbert series 
$h_{A_{r,s,a}}(u)$ (this will imply that these two series are actually the same). By Theorem \ref{classi}(ii) 
and Theorem \ref{CMnessrs}, The Hilbert series of $\Lambda_{r,s,q,t}$ and $\Lambda_{r,s,a}$ are the same, 
so it suffices to check that the Hilbert series of ${\rm Ker}\phi_{q,t}$ is dominated from below 
by the Hilbert series of ${\rm Ker}\phi_a$. 

To this end, let 
$$
D_{r+1,s+1,b}(\bold y,\bold z):=\prod_{j=1}^{r+1}\prod_{l=1}^{s+1}(y_j-z_l)(y_j-bz_l).
$$
Then any multiple of $D_{r+1,s+1,tq^{-1}}$ satisfies the quasi-invariance condition 
of Proposition \ref{quasi-inva}(i), so $D_{r+1,s+1,tq^{-1}}C\subset {\rm Ker}\phi_{q,t}$, giving the desired 
lower bound for the Hilbert series. 

\subsection{The quasi-invariant description of $A_{r,s,a}$ and $A_{r,s,q,t}$}

The construction of the algebras $A_{r,s,a}$ and $A_{r,s,q,t}$ implies that they can be described by quasi-invariance conditions on hyperplanes. 
Namely, we have the following result. 

\begin{proposition}\label{quasi-inva1} 
(i) For Weil generic $q,t$ the algebra $A_{r,s,q,t}$ is the algebra of polynomials 
$f(x,\bold y,\bold z)$ symmetric under $S_r\times S_s$ which satisfy the following quasi-invariance conditions: 

(1) $f(x; y_1,\dots,y_{r-1},u; z_1,\dots,z_{s-1},u)=f(u; y_1,\dots,y_{r-1},x; z_1,\dots,z_{s-1},x)$;

(2) $f(x; y_1,\dots,y_{r-1},u; z_1,\dots,z_{s-1},u)=f(x; y_1,\dots,y_{r-1},tu; z_1,\dots,z_{s-1},qu)$;

(3) $f(x; y_1,\dots,y_{r-1},tq^{-1}x; z_1,\dots,z_s)=f(q^{-1}x; y_1,\dots,y_{r-1},z_1,\dots,z_s)$;

(4) $f(x; y_1,\dots,y_r,z_1,\dots,z_{s-1},qt^{-1}x)=f(xt^{-1}; y_1,\dots,y_r; z_1,\dots,z_{s-1},x)$. 

(ii) For generic $a$ the algebra $A_{r,s,a}$ is the algebra of polynomials 
$f(x,\bold y,\bold z)$ symmetric under $S_r\times S_s$ which satisfy the following quasi-invariance conditions: 

(1) $f(x; y_1,\dots,y_{r-1},u,z_1,\dots,z_{s-1},u)=f(u; y_1,\dots,y_{r-1},x; z_1,\dots,z_{s-1},x)$;

(2) $((\partial_{y_r}-a\partial_{z_s})f)(x; y_1,\dots,y_{r-1},u; z_1,\dots,z_{s-1},u)=0$;

(3) $(((a+1)\partial_{y_r}-a\partial_x)f)(x; y_1,\dots,y_{r-1},x; z_1,\dots,z_s)=0$;

(4) $(((a+1)\partial_{z_s}-\partial_x)f)(x; y_1,\dots,y_r; z_1,\dots,z_{s-1},x)=0$. 

\end{proposition} 

\begin{proof} Let us prove (i). It is easy to check that conditions (1)-(4) (together with the $S_r\times S_s$-symmetry) 
are exactly the restriction of the quasi-invariance conditions of Proposition \ref{quasi-inva}(i) for $\Lambda_{r+1,s+1,q,t}$ 
to the hyperplane $y_{r+1}=z_{s+1}$ (i.e., they define the equivalence relation on points induced by restricting 
the relation of Proposition \ref{quasi-inva}(i) to this hyperplane). This implies the desired statement. 
The proof of (ii) is similar, using an infinitesimal version of this argument 
(as the equations (1)-(4) of (ii) are the infinitesimal versions of equations (1)-(4) of (i)).  
\end{proof} 

\section{The CM property of subspace arrangements of type $(2m,m,\dots,m)$.}

In this section we will use the same method as in the previous section to prove the following result 
about CM-ness of subspace arrangements, in the spirit of \cite{BCES}. 
Namely, for a partition $\lambda$ let $X_\lambda$ be the union of subspaces 
in $\Bbb C^{|\lambda|}$ defined by the condition that some $\lambda_1$ coordinates are the same, some other $\lambda_2$ coordinates 
are the same, etc. 

\begin{theorem}\label{2mmn} The variety $X_{(2m,m^{(r)})}$ is CM for any $r\ge 0$ and $m\ge 1$.   
\end{theorem} 

\begin{proof} Let $n=r+2$. Consider the variety $X_{m^{(n)}}$. Recall that $X_{m^{(n)}}$ is CM (\cite{EGL}, Proposition 3.11). 
The algebra $O(X_{m^{(n)}})$ can be viewed as a subalgebra of its normalization $O(\widetilde{X}_{m^{(n)}})$, a direct 
sum of polynomial rings. Let $I_{m^{(n)}}$ be the kernel of the morphism
$O(X_{m^{(n)}})\to O(X_{(2m, m^{(n-2)})})$, which we may again view as a module over $O(\widetilde{X}_{m^{(n)}})$.

\begin{lemma}\label{princid} $I_{m^{(n)}}$ is a principal ideal in $O(\widetilde{X}_{m^{(n)}})$.
\end{lemma} 

\begin{proof}  A point $\bold x=(x_1,x_2,...,x_{mn})$ with $x_1=\cdots=x_m,x_{m+1}=\cdots=x_{2m},\dots$ is in $X_{(2m,m^{(n-2)})}$ iff two of its $m$-blocks are equal, and thus a function in $I_{m^{(n)}}$ must be a multiple of the discriminant on each component in $\widetilde{X}_{m^{(n)}}$.  Conversely, since the discriminant on one component vanishes on all other components, we find that any function on $\widetilde{X}_{m^{(n)}}$ which is a multiple of the discriminant in each summand is actually in $I_{m^{(n)}}$.  It follows that the restriction of $I_{m^{(n)}}$ to each direct summand of $\widetilde{X}_{m^{(n)}}$ is the principal ideal generated by the discriminant, and thus $I_{m^{(n)}}$ is itself a principal ideal.
\end{proof} 

     Now, if we extend a generator of $I_{m^{(n)}}$ by a generic sequence of linear polynomials, the result will be a regular sequence, as it is regular in each direct summand of 
     $O(\widetilde{X}_{m^{(n)}})$. Let $B$ be the polynomial ring generated by the chosen sequence of linear polynomials.  Then (since a generic sequence of linear polynomials is a regular sequence for $X_{m^{(n)}}$, and since the latter is CM) we have a chain of free $B$-modules:
$$
I_{m^{(n)}}\subset O(X_{m^{(n)}})\subset O(\widetilde{X}_{m^{(n)}}),
$$
and thus a short exact sequence of $B$-modules of homological dimension 1:
$$
0\to O(X_{m^{(n)}})/I_{m^{(n)}}\to O(\widetilde{X}_{m^{(n)}})/I_{m^{(n)}}\to O(\widetilde{X}_{m^{(n)}})/O(X_{m^{(n)}})\to 0
$$
The middle term is free of finite rank since the algebra in the middle is CM (the function algebra on a disjoint union of hypersurfaces).
Thus, so is $O(X_{m^{(n)}})/I_{m^{(n)}}=O(X_{(2m,m^{(r)})})$. Hence, $X_{(2m,m^{(r)})}$ is a CM variety, as desired. 
\end{proof} 

On the basis of the results of \cite{BCES} and this paper, as well as computer calculations, we state the following conjecture. 

\begin{conjecture} 
$X_\lambda$ is CM if and only if one of the following holds: 

(1) $\lambda=(m^{(r)},1^{(s)})$ with $r\ge 1$, $m>s\ge 0$; 

(2) $\lambda=(2^{(r)},1^{(s)})$ for $r\ge 1,s\ge 0$; 

(3) $\lambda=(2m,m^{(s)})$, $m\ge 1$. 
\end{conjecture} 

Note that the ``if'' part of the conjecture is known, and only the ``only if'' part is in question. 

\section{$m$-quasi-invariants}
\label{mquasi}

\subsection{Rational $m$-quasi-invariants} 
Let $m\ge 1, r\ge 2,s\ge 1$ be integers. Following the paper \cite{FV2} (which treats the case $s=1$), define the algebra 
$\Lambda_{r,s}(m)$ to be the algebra of polynomials $P\in \Bbb C[y_1,...,y_r,z_1,...,z_s]$ 
which are symmetric in the $z_l$, satisfy the quasi-invariance conditions \eqref{clquas} for $a=m$,
and also the $m$-quasi-invariance condition:
\begin{equation}\label{clquas1}
P(\dots,y_j,\dots,y_k,\dots,z_1,\dots,z_s)-
P(\dots,y_k,\dots,y_j,\dots,z_1,\dots,z_s)
\end{equation} 
is divisible by $(y_j-y_k)^{2m+1}$
for $1\le j<k\le r$.

\begin{theorem}\label{mquasi-inv1} (M. Feigin)
If $m>s$ then the algebra $\Lambda_{r,s}(m)$ is CM.  
\end{theorem} 

\begin{proof}
Consider the algebra $B$ generated by $\Lambda_{r,s,m}$ and the deformed Calogero-Moser operator $L_2$. As follows from \cite{F,EGL,BCES}, 
this algebra is the quotient of the spherical rational Cherednik algebra ${\bold e}H_{1/m}(mr+s){\bold e}$ by a maximal ideal $I$. 
It is easy to see that $L_2$ preserves the space of polynomials satisfying \eqref{clquas1} (a calculation in codimension $1$ 
similar to the one in \cite{FV1}). Thus, $B$ acts naturally on $\Lambda_{r,s}(m)$. Hence, $\Lambda_{r,s}(m)$ is a module 
over the spherical Cherednik algebra ${\bold e}H_{1/m}(mr+s){\bold e}$ of minimal support. Therefore, by Theorem 1.2 of \cite{EGL}, 
$\Lambda_{r,s}(m)$ is a free module over $\Bbb C[Q_1,...,Q_{r+s}]$, hence it is a CM algebra, as claimed.  
\end{proof} 

Since characters of minimal support modules are explicitly known (see \cite{EGL}), 
the method of proof of Theorem \ref{mquasi-inv1} can be used to derive explicit formulas
for the Hilbert series of $\Lambda_{r,s}(m)$. In the appendix to this paper, M. Feigin 
uses these formulas to prove the conjecture from \cite{FV2} that the algebra $\Lambda_{r,1}(m)$ is Gorenstein.  

\begin{remark} 1. For $s=1$, Theorem \ref{mquasi-inv1} is proved in \cite{FV2}. 

2. Note that for $s=1$, Theorem \ref{classi} (i.e., Theorem 4.4 of \cite{BCES}) follows from Theorem \ref{mquasi-inv1} (proved in this case in \cite{FV2}) by interpolating with respect to $m$ 
(using the fact that the homogeneous components of $\Lambda_{r,s}(m)$ stabilize as $m\to \infty$, and its structure constants depend rationally on $m$).   
\end{remark} 

\subsection{Trigonometric (non-homogeneous) quasi-invariants} 
Let $\Lambda^{\rm trig}_{r,s}(m)$ be the algebra 
of polynomials $P\in \Bbb C[y_1,...,y_r,z_1,...,z_s]$ 
which are symmetric in the $z_l$ and satisfy
the trigonometric (non-homogeneous) $m$-quasi-invariance conditions:
$$
P(\dots,y_j+1,\dots,z_l-m,\dots)=P(\dots,y_j,\dots,z_l,\dots),\ y_j=z_l,
$$
for $1\le j\le r,1\le l\le s$, and 
\begin{equation}\label{clquas2}
P(\dots,y_j,\dots,y_k,\dots,z_1,\dots,z_s)-
P(\dots,y_k,\dots,y_j,\dots,z_1,\dots,z_s)
\end{equation} 
is divisible by $\prod_{p=-m}^m(y_j-y_k-p)$ for $1\le j<k\le r$. 

Note that the algebra $\Lambda^{\rm trig}_{r,s}(m)$ has a natural filtration by degree of 
polynomials. 

\begin{proposition}\label{mquasi-inv2} If $m>s$, 
we have ${\rm gr}(\Lambda^{\rm trig}_{r,s}(m))=\Lambda_{r,s}(m)$. 
In particular, the algebra $\Lambda_{r,s}^{\rm trig}(m)$ is CM.  
\end{proposition} 

\begin{proof} Consider the completion of 
the trigonometric Cherednik algebra 
${\bold e}H_{1/m}^{\rm trig}(mr+s){\bold e}$ near the identity element of the torus $(\Bbb C^\times)^{mr+s}$. 
This algebra has a decreasing filtration with associated graded isomorphic to 
${\bold e}H_{1/m}(mr+s){\bold e}$ (in fact, this deformation is known to be trivial). 
One can check that the action of the algebra ${\bold e}H_{1/m}(mr+s){\bold e}$
on $\Lambda_{r,s}(m)$ deforms to an action of ${\bold e}H_{1/m}^{\rm trig}(mr+s){\bold e}$  
on $\Lambda_{r,s}^{\rm trig}(m)$. Indeed, this amounts to checking that the rational deformed Macdonald-Ruijsenaars operator, 
i.e., the rational difference degeneration of 
the deformed Macdonald-Ruijsenaars operator (1) of \cite{SV2} preserves 
the non-homogeneous $m$-quasi-invariance conditions, which is done by a straightforward computation similar to the one in \cite{SV2}. 
Since the algebra $\Lambda_{r,s}^{\rm trig}(m)$ contains a principal ideal in $\Bbb C[y_1,...,y_r,z_1,...,z_s]^{S_s}$, 
the Hilbert series of the algebras ${\rm gr}(\Lambda_{r,s}^{\rm trig}(m))$ and $\Lambda_{r,s}(m)$ have the same asymptotics as $u\to 1$, i.e., 
give the same value at $1$ after multiplication by $(1-u)^{r+s}$ (namely, $1/s!$).  
Since $\Lambda_{r,s}(m)$ is a minimal support module over ${\bold e}H_{1/m}(mr+s){\bold e}$, this implies that 
we must have ${\rm gr}(\Lambda^{\rm trig}_{r,s}(m))=\Lambda_{r,s}(m)$ (as, because of equal growth, the quotient $\Lambda_{r,s}(m)/{\rm gr}(\Lambda^{\rm trig}_{r,s}(m))$ is a module over the rational Cherednik algebra with smaller support). 
\end{proof} 

\begin{remark} Let $R\subset \h$ be a root system with Weyl group $W$. For $\alpha\in R$ let $s_\alpha$ be the corresponding reflection. 
Let $m$ a multiplicity function on roots (see \cite{FV1}).
In this case we can define the ring of quasi-invariants $\bold Q_m\subset \Bbb C[\h]$, i.e. polynomials $f$ on the reflection representation $\h$ 
such that $f(x)-f(s_\alpha x)$ is divisible by $\alpha(x)^{2m_\alpha+1}$ for $\alpha\in R$, and the ring of trigonometric (non-homogeneous) quasi-invariants 
$\bold Q_m^{\rm trig}$, i.e. polynomials $f$ on $\h$ such that $f(x+\frac{1}{2}j\alpha^\vee)=f(x-\frac{1}{2}j\alpha^\vee)$ if $\alpha(x)=0$ for $j=1,...,m_\alpha$. 
Then one can use the same argument as in the proof of Proposition \ref{mquasi-inv2} (namely, the rational difference degeneration of \cite{Cha}, Proposition 2.1) to 
prove the following proposition:  

\begin{proposition}\label{graded} One has ${\rm gr}(\bold Q_m^{\rm trig})= \bold Q_m$. 
\end{proposition} 

In particular, this implies that $\bold Q_m^{\rm trig}$ is CM and, moreover, Gorenstein
(as by \cite{EG, BEG}, so is $\bold Q_m$).  

\begin{example} For the root system of type $A_1$ the rational Macdonald-Ruijsenaars operator has the form 
$$
(Mf)(x)=\frac{x-m}{x}(T-1)+\frac{x+m}{x}(T^{-1}-1),
$$
where $(Tf)(x)=f(x+1)$. It is easy to see that this operator preserves the space $\bold Q_m^{\rm trig}$ of polynomials $f$ 
such that $f(j)=f(-j)$ for $j=1,2,...,m$. The (completed) trigonometric Cherednik algebra acting on $\bold Q_m^{\rm trig}$ is generated 
by $M$ and $x^2$. Note that $M$ lives in filtration degrees $d\le -2$, and the degree $-2$ (leading) part of $M$ equals $\partial^2-\frac{2m}{x}\partial$, 
the rational Calogero-Moser operator for $A_1$. 
\end{example}  
\end{remark}

\begin{remark} Another proof of Proposition \ref{graded} can be obtained by using the rational difference degeneration $G_m^{\rm trig}:\Bbb C[\h]\to \Bbb C[\h]$ 
of Cherednik's shift operator (\cite{Ch, Cha}). More precisely, Corollary 8.28 of \cite{EG} proves that the image of the usual (differential) 
shift operator $G_m: \Bbb C[\h]\to \Bbb C[\h]$ is exactly $\bold Q_m$. Also, one can check that the image of $G_m^{\rm trig}$ is contained in 
$\bold Q_m^{\rm trig}$, which implies that ${\bold Q}_m^{\rm trig}$ is not smaller (i.e., the same size) as $\bold Q_m$, as desired.     
\end{remark} 
                     
                      \section{Appendix.          The Hilbert series of $\Lambda_{r,s}(m)$ }                             
    
    \begin{center}
    Misha Feigin
    \end{center}

    \vspace{5mm}
    
    {\it To Aleksandr Petrovich Veselov on the 60th birthday, with gratitude}
                                          
     \vspace{5mm}

In this Appendix we find Hilbert series of the algebra $\Lambda_{r,s}(m)$ introduced in Section~\ref{mquasi} assuming throughout that $m>s$. We also show that the algebra is Gorenstein if $s=1$. The algebra $\Lambda_{r,1}(m)$ is isomorphic to the algebra of quasi-invariants for the configuration ${\mathcal{A}_{r}(m)}$ considered in \cite{FV2}, \cite{CFV98}. The Gorenstein property of $\Lambda_{r,1}(m)$ was shown in \cite{FV2} for $r=2$ and it was conjectured to hold for any $r$.

Let $n=mr+s$. Let $\lambda$ be a partition of $n$. %We will also denote by $\lambda$ the irreducible $S_n$-module corresponding to the Young diagramme $\lambda$. 
Let $L_c(\lambda)$ be the corresponding irreducible module for the rational Cherednik algebra $H_c(S_n)$.  Let $e L_c(\lambda)$ be the corresponding irreducible module for the spherical subalgebra, $e=\frac{1}{n!}\sum_{w \in S_n} w$. For a partition $\tau$ of $r$ and a partition $\nu$ of $s$ we denote by $m \tau +\nu$ the partition of $n$ with terms $m \tau_i + \nu_i$. We will also denoted by $\tau$ the corresponding representation of $S_r$.

%$\tau=(\tau_1, \tau_2, \ldots)$ of $r$ and a partition $\nu= (\nu_1, \nu_2, \ldots)$ of $s$  we denote by $m \tau +\nu$ the partition $(m \tau_1+\nu_1, m \tau_2+\nu_2, \ldots)$ of $n$. 

%We will also need rational Cherednik algebra $H_c(S_n, \C^{n+1})$ where $S_n$ acts by permutations on the first $n$ coordinates in  $\C^{n+1}$, and the related simple modules $L_c(\tau, \C^{n+1})$. 

\begin{theorem}\label{quasi-decompose}
There is an isomorphism
$$
\Lambda_{r,s}(m) \cong \bigoplus_{\tau \vdash r}  \tau\otimes eL_{1/m}(m \tau+s)
$$
of ${\mathbb C}S_r \otimes eH_{1/m}(S_n)e$ modules. %where $\dim \tau$ is the dimension of the $S_r$ module corresponding to $\tau$.
\end{theorem}
\begin{proof}

It follows from the proof of Theorem \ref{mquasi-inv1}  that $\Lambda_{r,s}(m)$ is a module over $eH_{1/m}(S_n)e$. It follows from \cite{EGL} that as a module over $ {\mathbb C}S_r \otimes eH_{1/m}(S_n)e$  it can be decomposed as
\begin{equation}
\label{dec}
\Lambda_{r,s}(m) \cong  \bigoplus_{\nad{\tau \vdash r}{\nu \vdash s}} d_{\tau, \nu} \otimes eL_{1/m}(m \tau+\nu)
\end{equation}
for some ${\mathbb C}S_r$ modules $d_{\tau, \nu}$. Let us consider the localised module $\Lambda_{r,s}(m)_{\rm loc}$, where  localisation is at the powers of
$$
\alpha(x) = \sum_{w \in S_n} w \left( \prod_{\nad{1\le i \le mr}{mr+1\le j \le n}} (x_i -x_j)\right).
$$
It  is a module over the localised rational Cherednik algebra  $eH_{1/m}(S_n, U)e$, where $U \subset {\mathbb C}^n$ is given by
 $\alpha (x) \ne 0$. Equivalently, we localise quasi-invariants $\Lambda_{r,s}(m) \subset {\mathbb C}[y_1,\ldots,y_r, z_{1}, \ldots, z_s]$ with respect to the powers of 
$$
\widehat\alpha (y,z) =\prod_{\nad{1\le i \le r}{1\le j \le s}} (y_i-  z_j)^m.
$$

 Let $\Lambda'_{r,s}(m) \subset {\mathbb C}[y_1,\ldots, y_r, z_1,\ldots, z_{s}]$ consist of polynomials $p$ which are symmetric in $z$-variables and satisfy quasi-invariant conditions (\ref{clquas1}). It is a module over the spherical rational Cherednik algebra $e' H_{m, 1/m}(S_r \times S_s; {\mathbb C}^{r+s}) e'$, $e'=\frac{1}{r! s!} \sum_{w\in S_r \times S_s} w$. 
It follows from \cite{BEG} that as ${\mathbb C}S_r \otimes e' H_{m, 1/m}(S_r \times S_s; {\mathbb C}^{r+s}) e'$ module it decomposes as
 $$
 \Lambda'_{r,s}(m)=  \bigoplus_{\tau \vdash r}  \tau\otimes e_r L_{m}(\tau) \otimes e_s L_{1/m}(triv),
 $$
 where $e_r=\frac{1}{r!} \sum_{w\in S_r} w$, $e_s=\frac{1}{s!} \sum_{w\in S_s} w$.

  Consider localisation $\Lambda'_{r,s}(m)_{\rm loc}$ of the module $\Lambda'_{r,s}(m)$ at the powers of $\alpha'(y,z)=  \prod_{\nad{1\le i \le r}{1\le j \le s}} (y_i -z_j)$, which  is a module over the localised spherical  Cherednik algebra $e' H_{m,1/m}(S_r \times S_s, U') e'$, where  $U' \subset {\mathbb C}^{r+s}$ is given by $\alpha'|_{U'} \ne 0$.
   % and $eL_m (\tau; U')$ is the localisation of $eL_m (\tau; \C^{n+1})$. 
 It is clear that $\Lambda_{r,s}(m)_{\rm loc} \subseteq \Lambda'_{r,s}(m)_{\rm loc}$. Since for any $p \in \Lambda'_{r,s}(m)$ we have ${\a'}^t p \in \Lambda_{r,s}(m)$ for any $t\ge 2$ the opposite inclusion follows so these spaces are equal. 
 %that $L \Lambda_{r,s}(m) = L'\Lambda'_{r,s}(m)$.

It follows from the work \cite{W} that there is an isomorphism 
\begin{equation}
\label{dec2}
eL_{1/m}(m \tau+\nu)_{\rm loc} \cong (e_r L_{m}(\tau) \otimes e_s L_{1/m}(\nu))_{\rm loc},
\end{equation}
of $e' H_{m,1/m}(S_r \times S_s, U') e'$ modules, and that these modules are not isomorphic for different $(\tau, \nu)$. Since we localise at $S_r$-invariant elements $\widehat \alpha, \alpha'$ the decompositions \eqref{dec}, \eqref{dec2} imply that $d_{\tau, \nu}=0$ if $\nu$ has more than one part, and that $d_{\tau, s}\cong \tau$.
\end{proof}

%Now it follows from the decomposition \eqref{dec} and the work \cite{W}  that 
%$$
%\Lambda'_{r,s}(m) \cong \bigoplus_{\nad{\tau \vdash r}{\nu \vdash s}} d_{\tau, \nu} e_rL_{m}(\tau)\otimes e_sL_{1/m}(\nu),
%$$
%as a module over $e' H_{m,1/m}(S_r \times S_s; {\mathbb C}^{r+s}) e'$, where 

%Since elements of $\Lambda'_{r,s}(m)$ are invariant under $S_s$ we get that $d_{\tau, \nu}=0$ if the partition $\nu$ has more than one part.
%Thus
%$$
%\Lambda'_{r,s}(m) \cong \bigoplus_{\tau \vdash r} d_{\tau} e_rL_{m}(\tau)\otimes e_sL_{1/m}(triv),
%$$
%where $d_{\tau} = d_{\tau, s}$ 

%Similarly to \cite{BEG} where the case $s=0$ was considered it follows by further localisation at the discriminant $\prod_{1\le i <j \le r} (y_i-y_j)=0$ that $d_{\tau, s}=\dim \tau$ as required.

% It is established in \cite{BEG} that quasi-invariants $Q_{\mathcal{A}_{n-1}^{(m)}}$ are decomposed as
%$$
%Q_{\mathcal{A}_{n-1}^{(m)}} \cong \bigoplus_{\tau \vdash n} (\dim \tau) e L_m (\tau; \C^{n+1})
%$$
%as $e' H_m(S_n, \C^{n+1}) e'$-module, $e'=\frac{1}{n!} \sum_{w\in S_n}$. Therefore the localisation at powers of $\beta$ gives
%$$
%L'Q_{\mathcal{A}_{n-1}^{(m)}} \cong \bigoplus_{\tau \vdash n} (\dim \tau) e L_m (\tau; U')
%$$
%as modules over the localised algebra $e' H_{1/m}(S_n, U') e'$, where  $U' \subset \C^{n+1}$ is given by $\beta|_{U'} \ne 0$ and $eL_m (\tau; U')$ is the localisation of $eL_m (\tau; \C^{n+1})$.
%The statement now follows from the results of \cite{W}.

Let $s_\lambda$ be the Schur function corresponding to the partition $\lambda$. Define the coefficients $c_{\lambda; m}^\nu, b_{\lambda; m}^\nu$ by
\begin{equation}
s_\lambda(x_1^m, x_2^m, \ldots) = \sum_\nu c_{\lambda; m}^\nu s_{\nu}(x_1, x_2, \ldots),
\end{equation}
\begin{equation}
s_\lambda(x_1^m, x_2^m, \ldots) s_s (x_1, x_2, \ldots) = \sum_\nu b_{\lambda, s; m}^\nu s_{\nu} (x_1, x_2, \ldots).
\end{equation}
Let $\lambda$ be a partition of $r$. Define $\kappa(\lambda)=\sum_{1\le i < j \le r} s_{ij}|_\lambda$ the content of $\lambda$.
%, where $s_{ij}$ is the elementary transposition of $i$ and $j$.
Let $p_k(\lambda)$ be the multiplicity of the  representation $\lambda$ in the space of homogeneous polynomials of $r$ variables of degree $k$. Define the Hilbert series
$$
\chi_\lambda(t)= \sum_{k=0}^\infty p_k(\lambda).
$$
It is known from \cite{K} that
\begin{equation}
\label{chi}
\chi_\lambda(t)= \prod_{\square \in \lambda} \frac{t^{l(\square)}}{1-t^{h(\square)}},
\end{equation}
where $l(\square)$ is the leg length of a box, and $h(\square)$ is the hook length of a box.

Let $\Lambda^{(k)}_{r,s}(m)\subset \Lambda_{r,s}(m)$ be the subspace of homogeneous elements of degree $k$. Let
$$
P_{r,s;m}(t)= \sum_{k=0}^\infty \dim \Lambda^{(k)}_{r,s}(m) t^k
$$
be the Hilbert series of $\Lambda_{r,s}(m)$.

\begin{theorem}\label{thhilbser}
The Hilbert series of the algebra  $\Lambda_{r,s}(m)$ has the form
$$
P_{r,s;m}(t)= \sum_{\lambda \vdash r} \dim \lambda \sum_{\nu \vdash n} b^{\nu}_{\lambda, s; m} t^{\frac{ n (n-1) - 2 \kappa(\nu)}{2m}} \chi_\nu(t).
$$
\end{theorem}
\begin{proof}
It is shown in \cite{EGL} that in the Grothendieck group
$$
[L_{1/m}(m\lambda+s)]= \sum_{\nu \vdash n} b_{\lambda, s; m}^\nu [M_{1/m}(\nu)].
$$
Therefore
$$
[eL_{1/m}(m\lambda+s)]= \sum_{\nu \vdash n} b_{\lambda, s; m}^\nu [e M_{1/m}(\nu)],
$$
and hence (cf. \cite{EGL})
$$
Tr_{e L_{1/m} (m \lambda+s)} (t^h) = \sum_{\nu \vdash n} b_{\lambda, s; m}^\nu t^{\frac{n}{2}-\frac{\kappa(\nu)}{m}} \chi_\nu(t),
$$
where $h=\frac12 \sum_{i=1}^n (x_i \nabla_i + \nabla_i x_i)$ is the scaling element of the rational Cherednik algebra.

On the other hand the action of the operator $h$ in the polynomial representation ${\mathbb C}[x_1,\ldots,x_n]$ is given by 
$$
h=\sum_{i=1}^n x_i \p_{x_i} +\frac{n}{2}-\frac{1}{m}\sum_{i<j}^n s_{ij},
$$
which reduces to 
$
h^{res}=\sum_{i=1}^n x_i \p_{x_i} +\frac{n}{2}-\frac{n(n-1)}{2m}
$ 
on $S_n$-invariants. Its action in the representation 
$\Lambda_{r,s}(m)\subset {\mathbb C}[z_1,\ldots, z_{r}, y_1, \ldots, y_s]$  is by the same differential operator $h^{res}$ where the Euler field component 
$\sum_{i=1}^n x_i \p_{x_i} $ acts as its restriction  to the Euler operator in $y,z$-plane which is $\sum_{i=1}^r y_i \p_{y_i} + \sum_{i=1}^s z_i \p_{z_i}$. 
%we have $h= \sum_{i=1}^{n+1} z_i \partial_{z_i}+\frac{N}{2}- \frac{N(N-1)}{2m}$. 
The statement now follows from Theorem~\ref{quasi-decompose}.
\end{proof}

Let us now consider the case $s=1$ so $n=mr+1$. We will derive another formula for the Hilbert series of the quasi-invariants $\Lambda_{r}(m):=\Lambda_{r,1}(m)$. It is based on the following results from \cite{EGL}. Let $\lambda$ be a partition of $r$. For the representations of rational Cherednik algebra  $H_{1/m}(S_{mr})$ one has
\begin{equation}
[L_{1/m}(m \lambda)] = \sum_{\nu \vdash mr} c_{\lambda; m}^\nu [M_{1/m}(\nu)].
\end{equation}
Then it is shown in \cite{EGL} that
\begin{equation}
[L_{1/m}(m \lambda+1)] = \sum_{\nu \vdash mr} c_{\lambda; m}^\nu [F M_{1/m}(\nu)],
\end{equation}
where the functor $F: H_{1/m}(S_{mr})-mod \to H_{1/m}(S_{mr+1})-mod$ acts on the standard modules as follows.
Let $\nu$ be a partition of $m r$. Then in the Grothendieck groups
$$
F: [M_{1/m}(\nu)] \to \bigoplus_{\widehat \nu \in B_\nu} [M_{1/m}(\widehat \nu)],
$$
where each diagram in the set $B_\nu$ is obtained from the diagram $\nu$ by adding a box with the content congruent to 0 modulo $m$.

This allows to restate Theorem \ref{thhilbser} in the following form.

\begin{corollary}
The Hilbert series of the algebra $\Lambda_r(m)$ has the form
\begin{equation}\label{quasi-form-2}
P_{r;m}(t)= \sum_{\lambda \vdash r} \dim \lambda \sum_{\nu \vdash m r} \sum_{\widehat \nu \in B_\nu} c^{\nu}_{\lambda; m} t^{\frac{r n}{2} - \frac{\kappa(\widehat\nu)}{m} } \chi_{\widehat\nu}(t).
\end{equation}
\end{corollary}
Note that the right-hand side of the series \eqref{quasi-form-2} may contain  fractional powers of $t$ which would have to cancel.

It is established in \cite{FV2} that the graded algebra $\Lambda_{r}(m)$ is Cohen-Macaulay. It is convenient to use the form \eqref{quasi-form-2} to show that the algebra $\Lambda_{r}(m)$  is Gorenstein. 

\begin{theorem}
The Hilbert series of the algebra of quasi-invariants $\Lambda_{r}(m)$  satisfies the symmetry property
$$
P_{r; m}(t^{-1})=(-1)^{r+1} t^{n(1-r)} P_{r; m}(t).
$$
\end{theorem}
\begin{proof}
Let us choose a term in the sum \eqref{quasi-form-2} corresponding to the diagrams $\lambda, \nu, \widehat \nu$. Notice that for the conjugate diagrams $\lambda', \nu'$ one can choose $\widehat {\nu'} = {\widehat \nu}'$. Indeed, if $\widehat\nu$ is obtained from $\nu$ by adding a box with the content $k$ then the transposed partition ${\widehat\nu}'$ is obtained from  ${\nu}'$ by adding a box with the content $-k$ so both contents are congruent to 0 modulo $m$ and ${\widehat \nu}' \in B_{\nu'}$. Thus the series \eqref{quasi-form-2} decomposes as a sum of terms of the form
$$
f(t)=(\dim \lambda) c_{\lambda; m}^\nu t^{\frac{r n}{2} - \frac{\kappa(\widehat \nu)}{m}} \chi_{\widehat \nu}(t)+
(\dim \lambda') c_{\lambda'; m}^{\nu'} t^{\frac{r n}{2} - \frac{\kappa({\widehat \nu}')}{m}} \chi_{{\widehat \nu}'}(t).
$$
Recall that $\dim \lambda = \dim \lambda'$ and $c_{\lambda; m}^\nu= (-1)^{(m-1)r} c_{\lambda'; m}^{\nu'}$ (see \cite{EGL}).
It is also easy to see from \eqref{chi} that
$$
\chi_{\widehat \nu}(t)=(-1)^{n} t^{-n}\chi_{\widehat \nu'}(t^{-1}).
$$
Therefore
$$
f(t)=(\dim \lambda) c_{\lambda; m}^{\nu} t^{\frac{r n}{2}}\left(t^{-\frac{\kappa(\widehat \nu)}{m}} \chi_{\widehat \nu}(t) + (-1)^{(m-1)r}t^{\frac{\kappa(\widehat \nu)}{m}} \chi_{{\widehat \nu}'}(t)\right),
$$
and
\begin{multline*}
f(t^{-1})= (\dim \lambda) c_{\lambda; m}^{\nu} t^{-\frac{r n}{2}}\left(t^{\frac{\kappa(\widehat \nu)}{m}} (-1)^{n} t^{n} \chi_{{\widehat \nu}'}(t) + (-1)^{(m-1)r}t^{-\frac{\kappa(\widehat \nu)}{m}} (-1)^{n} t^{n}\chi_{{\widehat \nu}}(t)\right)  \\
= (-1)^{r+1} t^{n(1-r)} f(t),
\end{multline*}
so the statement follows.

\end{proof}

By Stanley criterion \cite{S} we have the following

\begin{corollary}
The algebra $\Lambda_{r}(m)$ is Gorenstein. 
\end{corollary}

We are going to obtain yet another form of the Hilbert series \eqref{quasi-form-2}.                                  
                Note that the coefficients $c^\nu_{\lambda; m}$ can be expressed in terms of characters of the symmetric group. Let $\mu$ be a partition of $r$ and denote by $C_{\mu}$ the corresponding conjugacy class in $S_r$. Then
$$
c^\nu_{\lambda; m} = \sum_{\mu \vdash r} \frac{|C_\mu|}{r!} \chi^\lambda(C_\mu) \chi^\nu(C_{m \mu}),
$$
where $\chi^\lambda$, $\chi^\nu$ are characters of representations of $S_r, S_{m r}$ corresponding to the partitions $\lambda, \nu$ (see e.g. \cite{LZ}).

Let $\widehat \chi^\lambda$ be the character of the module $U_\lambda$ which is induced from the trivial one for the parabolic subgroup corresponding to partition $\lambda$. Recall the Kostka matrix      $K_{\mu \lambda}$  given by the relations  $\widehat \chi^\lambda = \sum_\mu K_{\mu \lambda} \chi^\mu$.       We will also need the  inverse Kostka matrix $K^{-1}$  satisfying $\chi^\lambda = \sum_\mu K_{\mu \lambda}^{-1} \widehat \chi^\mu$.                                   
 Then we have  
 \begin{equation}
 \label{stst}
 \sum_{\lambda \vdash r} \dim \lambda \, c_{\lambda; m}^\nu = \sum_{\nad{\lambda, \mu \vdash r}{\tilde \nu \vdash m r}} \dim \lambda \frac{|C_\mu|}{r!} \chi^\lambda (C_\mu) K^{-1}_{\tilde \nu \nu} \widehat\chi^{\tilde \nu}(C_{m \mu}).
 \end{equation}
 Note that $\widehat\chi^{\tilde \nu}(C_{m \mu})$ is non-zero only if partition $\tilde \nu$ has the form $\tilde \nu = m \alpha$ for some $\alpha\vdash r$ in which case 
 $\widehat\chi^{\tilde \nu}(C_{m \mu}) = \widehat \chi^\alpha(C_\mu)$. Taking into account  orthogonality of characters we continue \eqref{stst} as
 $$
 \sum_{\lambda,\mu,\alpha,\beta\vdash r} \dim \lambda \frac{|C_\mu|}{r!}  \chi^\lambda (C_\mu) K^{-1}_{m\alpha, \nu} 
 K_{\beta \alpha} 
\chi^{\beta}(C_{\mu})  =
\sum_{\alpha, \lambda \vdash r} \dim \lambda \, K^{-1}_{m\alpha, \nu} K_{\lambda \alpha}
 $$
$$
= \sum_{\alpha \vdash r} \dim U_\alpha \, K^{-1}_{m\alpha, \nu} = \sum_{\alpha \vdash r} \frac{r!}{\alpha!} K^{-1}_{m\alpha, \nu},
$$                                          
where $\alpha!= \a_1! \a_2!\ldots$.                              
Thus we get the following expression for the Hilbert series \eqref{quasi-form-2}:
\begin{equation}\label{quasi-form-3}
P_{r;m}(t)= \sum_{\nad{\alpha \vdash r, \nu \vdash mr}{\widehat \nu \in B_\nu}} \frac{r!}{\alpha!} K^{-1}_{m\alpha, \nu} t^{\frac{r n}{2} - \frac{\kappa(\widehat\nu)}{m} } \chi_{\widehat\nu}(t).
\end{equation}
It would be interesting to see if there is a simpler form of the Hilbert series $P_{r;m}(t)$.

Finally we note that the algebra $\Lambda_{r,s}(m)$ is not expected to be Gorenstein for $s>1$ as the case $r=1$ shows. Indeed, it is shown in \cite{J} that for any non-zero $m$ the Hilbert series $P_{1,s; m}$ is the same which is known from \cite{SV1} to be equal to $h=\frac{1-t+t^{s+1}}{(1-t)^2(1-t^2)\ldots(1-t)^s}$ so the algebra is not Gorenstein.

{\bf Acknowledgements.}      I am very grateful to P. Etingof for a number of useful and helpful discussions and comments, to I. Losev for explanations about \cite{EGL}, to C. Korff,  J. Nimmo, and A.P. Veselov   for useful discussions. I would also like to thank V. Lunts for the  hospitality at his summer seminar 2013 where a  part of this work was done.

\end{document}